\documentclass{amsart}
\usepackage{graphicx}
\usepackage{amssymb}
\newtheorem{thm}{Theorem}[section]
\newtheorem{cor}[thm]{Corollary}
\newtheorem{lemma}[thm]{Lemma}
\newtheorem{ex}[thm]{Example}

\newtheorem{prop}[thm]{Proposition}
\theoremstyle{definition}

\newtheorem{remark}[thm]{Remark}
\theoremstyle{question}

\theoremstyle{Conjecture}

\numberwithin{equation}{section}

\begin{document}

\title[ criteria for nilpotency of groups via  partitions]{ criteria for nilpotency of  groups via  partitions}%
\author{L. J. Taghvasani  and M. Zarrin}%

\address{Department of Mathematics,  University of Kurdistan, P.O. Box: 416, Sanandaj, Iran}
\email{L.jafari@sci.uok.ac.ir}
\address{Department of Mathematics, University of Kurdistan, P.O. Box: 416, Sanandaj, Iran}
 \email{Zarrin@ipm.ir}
\begin{abstract}
Let $G$ be a group and $T< G$. A set $\Pi=\{H_1, H_2, \dots , H_n\}$ of proper subgroups of $G$ is said to be a $\mathit{strict}$ $\mathit{T}$-$\mathit{partition}$ of $G$, if $G=\cup _{i=1} ^{n} H_i$ and $H_i \cap H_j=T$ for every $1\leq i, j \leq n$. If $\Pi$ is a strict $T$-partition of $G$ and the orders of all components of $\Pi$ are equal, then we say that $G$ has an $ET$-partition. Here we show that:
A finite group $G$ is nilpotent if and only if every subgroup $H$ of $G$ has  an ES-partition, for  some $S \leq H$.
\end{abstract}
\maketitle

\section{Introduction}

Let $G$ be a finite group and $S< G$. A cover for a group $G$ is a collection of subgroups of $G$ whose union is $G$.  We use the term $n$-cover for a cover with $n$ members.
A cover $\Pi =\{H_1, H_2, \dots, H_n\}$ is said to be a  strict $\mathit{S}$-partition of $G$ if $H_i\cap H_j= S$ for $i\neq j$ and $\Pi$ is said an equal strict $\mathit{S}$-partition (or $ES$-partition ) of $G$, if $\Pi$ is a strict $S$-partition and $|H_i|=|H_j|$ for all $i\neq j$. If $S$ is the identity subgroup and $G$ has a strict $S$-partition (equal strict $\mathit{S}$-partition), then we say that $G$ has a partition (equally partition, resp.).

In 1906 Miller pioneered the study of  finite groups with a partition. Subsequently in 1961, Baer, Kegel and Suzuki completed the classification of partitions of finite groups, as follows.

\begin{thm} \cite{baer, suzuki, kegel}\label{class}
Let $G$ be a finite group. Then $G$ admits a nontrivial partition if and only if $G$ is isomorphic with one of the following groups.

 \begin{description}

 \item $G$ is a $p$-group with $H_p(G)\neq G$ and $|G|> p$, where $H_p(G)$ is the Hughes subgroup of $G$. That is the subgroup generated by all the elements of $G$ whose order is not $p$;
 \item  $G$ is a Frobenius group;
 \item $G$ is a group of Hughes-Thompson type, that is a non $p$-group with $H_p(G)\neq G$ for some prime $p$;
 \item  $G$ is isomorphic with $PGL(2, p^h)$, p is an odd prime;
 \item $G$ is isomorphic with $PSL(2, p^h)$, p is a prime;
 \item $G$ is isomorphic with a Suzuki group $G(q)$, $q=2^h$, $h> 1$.

 \end{description}

\end{thm}

In 1966  \cite{isaacs}, Isaacs proved that the $\mathbf{only}$ finite groups which can be partitioned by subgroups of equal orders (or,  only finite groups which has an equally partition) are the finit $p$-groups of exponent $p$ (we will mention this result in throughout the paper as Isaacs' Theorem).
Zappa in \cite{zappa2},  reduced the problem of determining the strict $S$-partitions to the analogous problem for partitions. In fact, he investigated finite groups with a strict $S$-partition, where $S$ is an antinormal subgroup of $G$. We say that a subgroup $T$ is antinormal in $G$, when no normal subgroup $1\neq N$ of $G$ is contained in $T$.

Here  we  give a new characterization of finite non-cyclic  nilpotent groups by $ES$-partitions.

\begin{thm} \label{T1}
For any  finite non-cyclic group $G$, the following statements are equivalent.

 \begin{itemize}

  \item  $G$ is a nilpotent group;
   \item  $G$ has an $ES$-partition  such that  $S\unlhd G$ and $S\leq Z_i(G)$, for some $i\in \mathbb{N}$;
 \item  Every subgroup $H$ of $G$ has  an $ES$-partition, for  some $S \leq H$.

 \end{itemize}

\end{thm}

Let $\mathcal{Y}$ be a class of groups. A group $G$ is a minimal non-$\mathcal{Y}$-group (or a $\mathcal{Y}$-critical group), if $G\not\in \mathcal{Y}$ but all proper subgroups of $G$ belong to $\mathcal{Y}$. It is clear that detailed knowledge of the structure of minimal non-$\mathcal{Y}$-groups can provide insight into what makes a group belong to $\mathcal{Y}$. For instance, minimal non-nilpotent groups (called Schmidt groups) were analyzed by Schmidt \cite{schm}, and proved that such groups are solvable (see also \cite{zar}). We say that a group $G$ is a minimal non-$ES$-partition group if $G$ has no $ES$-partition but every proper subgroup $H$ of $G$ has an $ET$-partition, for some $T\leq H$. As a Corollary of the above Theorem we obtain the new characterization for Schmidt groups. In fact we prove that (see Corollary \ref{co35}, below):
  $$\text{A ~~group~~}  G \text{~~is~~ a ~~minimal~~ non-ES-partition group~~if~and~only~if ~}$$  $$G \text{~~is~~ a ~~minimal ~~non-nilpotent ~~ group.}$$

  Finally, in view of  the above result, the following question arises naturally: Assume that $F$ is a nontrivial subgroup of  $G$ such that $G$ has no an  $EF$-partition but every proper subgroup of $G$ has  an $EF$-partition. What we can say about the structure of such groups?  We show that  these groups are  minimal non-cyclic groups (except  $Z_p\times Q_8$ and $Z_3\ltimes Q_8$, where $Q_8$ is the quaternion group, $Z_p$ is a cyclic group of prime order $p$ and $Z_3\ltimes Q_8$ is semidirect product $Q_8$ by $Z_3$).  The classification of minimal non-cyclic groups, is known (see \cite{MM}). More precisely the following holds.

  \begin{thm} \label{T2}
  Assume that $G$ is a group and $F$ is a nontrivial subgroup of  $G$. Then   every proper non-cyclic subgroup of $G$ has  an $EF$-partition if and only if  G is isomorphic to one of the following groups:
  \begin{itemize}

  \item  $Z_p \times Z_p$ with $p$ prime;
   \item  $Q_8$, the quaternion group of order 8;
    \item  $Z_p \times Q_8$ and $Z_3\ltimes Q_8$;
 \item  $ < a, b ~|~ a^p=b^{q^{m}}=1; b^{-1}ab=a^r >$, with $p, q$ distinct primes and $r \not\equiv 1 (mod ~p)$ and $r^q \equiv 1 (mod ~p)$.

 \end{itemize}

\end{thm}

\section{Proofs}

 Zappa proved the following two theorems. In order to prove the main results we need the following theorems.

\begin{thm}{\rm (}\cite{zappa2}{\rm )}\label{zapa1}
Let $G$ be a group and $S$ a subgroup of $G$. Let $N$ be a normal subgroup of $G$ such that $N\leq S$.
Then the set $\{H_1, H_2, \dots , H_r\}$ is a strict $S$-partition of $G$ if and only if $\{H_1/N, H_2/N, \dots , H_r/N\}$ is a strict $S/N$-partition of $G/N$.
\end{thm}

\begin{thm}{\rm (}\cite{zappa2}{\rm )}\label{zapa2}
Let $G$ be a group, $S$ an antinormal subgroup of $G$, and $H_1, H_2, \dots,$$ H_r$ be a set of subgroups of $G$ with $S\leq H_i$ for every $1\leq i \leq r$. Then the following conditions are equivalent:

\begin{itemize}
\item $H_1$, $H_2$, $\dots$, $H_r$ is a non trivial strict $S$-partition of $G$.

\item $G$ is a Frobenius group and $S$ a cyclic Frobenius complement of $G$. The Frobenius kernel $K$ is a $p$-group, $H_i=SK_i$, where $\{K_1,\dots, K_r \}$ is a non trivial partition of $K$.
\end{itemize}
\end{thm}

\begin{lemma}\label{lr1}
Assume that  $G$ is a group which has an $ES$-partition. Then $$ G \text{~~is~~ nilpotent~if~and~only~if~~~~~~}  \exists ~~ i\in \mathbb{N} \text~~{such ~~that ~~} S\leq Z_i(G) \text{~~and~~} S\unlhd G.$$
\end{lemma}
\begin{proof}
Let $\Pi =\{H_1, \dots , H_n\}$ be an $ES$-partition for $G$, $S\leq Z_i(G)$ for some $i\in \mathbb{N}$ and $S\unlhd G$. Then the set $\{H_1/S,\dots, H_n /S \}$ is an equally partition for $G/S$. Therefore by Isaacs' Theorem, $G/S$ is a $p$-group of exponent $p$. Hence $G$ is nilpotent.

Now assume that $G$ is nilpotent group and has an $ES$-partition. Clearly, $\exists ~~i\in \mathbb{N}$ such  that $S\leq Z_i(G)$.
So we it is enough to show that $S\unlhd G$. If $S$ is antinormal subgroup of $G$, then by Theorem \ref{zapa2}  $G$ is a Frobenius group so $Z(G)=1$ which
contradicts to nilpotency of $G$. So $S$ is not antinormal  and then there exists $1\neq N\unlhd G$ such that $N\leq S$. Let $N=S_G$ be the core of $S$ in $G$. If $S_G< S$,
 then the subgroup $S/S_G$ of $G/{S_G}$ is antinormal subgroup of $G/{S_G}$ and by Theorem \ref{zapa1}, $\{ H_1/ S_G , H_2/S_G, \dots , H_r/S_G\}$ is a strict $S/S_G$-partition
  of $G/S_G$. Now again by Theorem \ref{zapa2}, $G/{S_G}$ is Frobenius group which contradicts with nilpotency of $G/S_G$. Thus $S=S_G$.
\end{proof}

\begin{remark}
In Lemma \ref{lr1}, the condition of equality of orders of components of $\Pi$ is necessary. In particular, there are some non-nilpotent groups like  $G$ such that $G$ has a strict $Z(G)$-partition (not necessarily equal). For example consider the dihedral group $D_{2n}=\langle x, y \mid  x^n =y^2 =1 , x^y=x^{-1} \rangle $, where $n=2k$ is even.  Then $Z(D_{2n})=\langle x^k \rangle$, $C_{D_{2n}}(x)=\langle x \rangle$ and for every $1\leq i\leq n$, $C_{D_{2n}}(x^iy)=\{e, x^iy, x^k, x^{k+i}y\}$. Now the set of centralizers of $G$ forms an strict $Z(G)$-partition for $G$ and $D_{2n}$ is not nilpotent in general.
\end{remark}

\begin{lemma}
Let $G$ be a finite group which has an $ES$-partition
 $\{H_1, \dots , H_n\}$. Then for every $1\leq i\leq n$, $\pi(H_i)=\pi(G)$.
\end{lemma}
\begin{proof}
Let $p\in \pi(G)$ and $g\in G$ with $|g|=p$. Then there exists $H_i$ such that $g\in H_i$. Thus $\pi(G)\subseteq \cup_{i=1}^{n} \pi(H_i)$ and so $\pi(H_j)=\pi(H_i)$ for every $1\leq i, j\leq n$, since $|H_i|=|H_j|$.  Hence for every $1\leq i\leq n$, $\pi(H_i)=\pi(G)$.
\end{proof}

\begin{cor}\label{co1}
Let $G$ be a finite group with square-free order. Then $G$ has no $ES$-partition and so every square free subgroup of every minimal non-$ES$-partition group is cyclic.
\end{cor}

 \begin{lemma}\label{semi}
Assume that  $G$ is a finite group which has an $ES$-partition, say $\Pi=\{H_1, H_2, \dots , H_n\}$. If $G$ acts on the group $H$, then the semidirect product $G\ltimes H$ has an $E(S\ltimes H)$-partition.
\end{lemma}
\begin{proof}
Since $G$ acts on $H$, so every $H_i$ acts on $H$ and we can form the semidirect product $H_i\ltimes H$. On the other hand
\begin{align*}
G\ltimes H =(\cup_{_{i=1}} ^nH_i) \ltimes H=\cup _{i=1} ^n (H_i \ltimes H)
\end{align*}
and
\begin{align*}
(H_i\ltimes H)\cap (H_j\ltimes H)= (H_i \cap H_j) \ltimes H =S\ltimes H
\end{align*}
So the set $\{H_i \ltimes H, \dots , H_n\ltimes H\}$ is an $E(S\ltimes H)$-partition for $G\ltimes H$.
\end{proof}

\begin{lemma}\label{lr2}
Let $G$ be a finite group. Then
$G$ is nilpotent if and only if every subgroup $H$ of $G$  has  an $ES$-partition, for some $S \leq H.$
\end{lemma}
\begin{proof}
If $G$ is a nilpotent group, then by Lemma \ref{semi} it is enough to show that every finite $p$-group, $P$ has an $ES$-partition. Now as $P/\phi(P)$ is $p$-elementary  so by Isaacs' Theorem  and Theorem \ref{zapa1},  we can obtain that  $P$ has an $E(\phi(P))$-partition (see also Proposition \ref{p32}, below).

Now suppose that for every non-cyclic subgroup $H$ of $G$ has an $ES$-partition, for a subgroup $S$ of $H$. Now, by  induction on the order of $G$, we deduce that all proper subgroups of $G$ are nilpotent. That is, $G$ is an Schmidt group. Therefore $G=P\ltimes Q$, where $P$ is a normal Sylow p-subgroup and Q is a cyclic Sylow q-subgroup $G$.
Let $\{H_1, \dots, H_t\}$ be an $ES$-partition of $G$. We claim that $S\lhd G$. Suppose by contrary that $S$ is not normal in $G$. First assume that $S$ is antinormal.  Then  Theorem \ref{zapa2} implies that $G$ is a Frobenius group and so $$1=Z(G)=\phi(G)=\phi(P)\times <y^q>.$$
It follows that  $P$ is an elementary abelian Sylow p-subgroup and $Q$ is cyclic of order $q$.   Now let $x\in S$ and $y\in Q$, where $|x|=p$ and $|y|=q$, then there exists $1\leq i\leq t$ such that $y\in H_i$. Now as  $x\in S\leq H_i$ for every $1\leq i\leq t$. Hence $x, y \in H_j$. By assumption, $H_j$ is nilpotent and $(|x|,|y|)=(p,q)=1$, so $xy=yx$, it means that $S$ doesn't act on $Q$ fixed-point-freely, which is a contradiction, since $G$ is Frobenius.
So we may assume that $S_G> 1$, we show that $S_G=S$. If $S_G< S$, then by Zappa's theorem, $G/S_G$ has an $ES/S_G$-partition,  and $S/S_G$ is an antinormal subgroup of $G/S_G$. Hence $G/S_G$ is a Frobenius group which is Schmidt as well. Hence $G/S_G=S/S_G\ltimes QS_G/S_G$ and $|S/S_G|=p$ for some prime $p$. Since $G/S_G$ is Frobenius, so $S/S_G$ must act on $QS_G/S_G$ fixed-point freely. That is, for every $x\in S\setminus S_G$ and $y\in Q\setminus S_G$, $[x,y]S_G=[xS_G,yS_G]\neq S_G$. But we will show that there exists $x\in S\setminus S_G$ and $y\in Q\setminus S_G$ and $|x|=p^i$ and $|y|=q^j$ such that $[x,y]=1$. It follows that $[xS_G, yS_G]=S_G$, and this contradiction shows that $S=S_G$, as wanted. First note that since $QS_G/S_G$ is the kernel of $G/S_G$, so it is non-trivial and there exists $y\in Q\setminus S_G$.  Now as $|S/S_G|=p$, we can conclude that every $q$-element of $S$ belongs to $S_G$ and there exists a $p$-element $x$ belongs to  $S\setminus S_G$.  Let $x\in S\setminus S_G$ and $|x|=p^i$. Therefore $x\in S\leq H_i$ for every $1\leq i\leq t$. Hence there exists $j$ such that $x,y\in H_j$ and  $[x,y]=1$, since $H_j$ is nilpotent.

Therefore $S\lhd G$. Now as $G$ is an Schmidt group and $G=P\ltimes Q$,  Isaacs' Theorem implies that  $G/S$ is a $r$-group of exponent $r$, where $r\in \{p, q\}$. Hence $P\leq S$ or $Q \leq S$. If $P\leq S$, then, as $S$ is nilpotent, we can follow that   $P$ is a normal subgroup of $G$ and so $G$ is nilpotent, a contradiction. If $Q\leq S$, then for every $i$, $Q\leq H_i$ and  so $q^m \mid |H_i|$, for every $i$. On the other hand, $P=\langle g\rangle$ is cyclic and $g\in H_i$ for some $i$.  Now as $|H_i|=|H_j|$ for every $j$, so $p^n \mid |H_i|$. Thus $|G|=|H_i|$, a contradiction.
\end{proof}

\begin{cor}\label{co35}
A group $G$ is a minimal non-$ES$-partition group if and only if  $G$ is a Schmidt group.
\end{cor}

Note that  Lemmas \ref{lr1} and \ref{lr2}, complete the proof of the main Theorem. But according to Lemma \ref{lr1}, we can see that if a group $G$ has an $ES$-partition, then the properties of
the subgroup $S$ has influence on the structure of $G$.
Therefore we investigate the influence of some properties of the subgroup $S$ (such as, abelian normal, nilpotent normal and solvable subgroup) on the group $G$ which has an $ES$-partition.
Here, we  give a new characterization of finite non-cyclic $p$-groups by whose equal strict partitions (see also Theorem 3.2 of \cite{Afp}).

\begin{prop}\label{p32}
Let $G$ be a non-cyclic finite group. Then $G$ has an $E(\phi(P))$-partition  if and only if $G$  is a $p$-group.
\end{prop}
\begin{proof}
Let $G$ be a non-cyclic finite group and $\Pi=\{H_1, H_2, \dots , H_r\}$ be an $E(\phi(G))$-partition for $G$. Then by Theorem \ref{zapa1}, $G/ \phi(G)$ has an equal partition. Therefore, by Isaacs' Theorem,
$G/ \phi(G)$ is a $p$-group of exponent $p$. Now it is well-known that for a finite group $G$,
the  prime divisors of
$|G/ \phi(G)|$ is equal to the prime divisors $|G|$. Hence $G$ is a $p$-group.

Assume that $G$ is a non-cyclic finite $p$-group. Then $G/\phi(G)$ is an elementary abelian $p$-group so, by Isaacs' Theorem,  $G/\phi(G)$  has an equally partition and, by  Theorem \ref{zapa1},  $G$ has an
$E(\phi(G))$-partition.
 \end{proof}
 It seems likely that (in view of the above results) every finite group which has an $ES$-partition where $S$ is a abelian normal subgroup, should be abelian. But   we show that there exist non-nilpotent finite groups that has a $ES$-partition, where $S$ is an abelian normal subgroup of $G$.  Consider the following example.

\begin{ex}
Let $A=Z_3\times Z_3\times Z_3$, $B=Z_2\times Z_2$, $b_1=(1,0)$, $b_2=(0,1)$ and $b_3=(1,1)$. Now define an action of $B$ on $A$ by setting
\begin{align*}
(x_1,x_2,x_3)^{b_1}=(-x_1,-x_2,x_3)\\
(x_1,x_2,x_3)^{b_2}=(-x_1,x_2,-x_3)\\
(x_1,x_2,x_3)^{b_3}=(x_1,-x_2,-x_3).
\end{align*}
It is easy to see that the three subgroups $A\langle b_1\rangle$, $A\langle b_2\rangle$, $A\langle b_3\rangle$ are isomorphic and are a cover for group $G=A\rtimes B$,  while $G$ is non-nilpotent.
\end{ex}

\begin{prop}\label{0}
Let $G$ has an $ES$-partition, where $S$ is a solvable subgroup of $G$, then $G$ is solvable.
\end{prop}
\begin{proof}
If $S$ is antinormal, then by \cite{zappa2} $S$ is cyclic and $G$ is Frobenius group with complement $S$. So $G$ is solvable. Otherwise if $S$ is not antinormal, then there is a normal subgroup $1\neq N$ of $G$ such that $N\leq S$. If $N< S$, then since $G$ has a strict $S$-partition, so by  Theorem \ref{zapa1}, $G/N$ has strict $S/N$-partition and by induction on the order of $G$, $G/N$ is solvable, so $G$ is solvable. Otherwise, if $N=S$, then $G/N$ has an equal partition and by Isaacs' Theorem, $G/S$ is a $p$-group and so $G$ is solvable.
\end{proof}

Finally, we prove Theorem \ref{T2}. \\\\
 $\mathbf{The~~ proof~~ of~~ Theorem \ref{T2}.}$ Assume that $F$ is a nontrivial subgroup of  group $G$ such that every proper non-cyclic subgroup of $G$ has  an $EF$-partition and $G$ is not minimal non-cyclic group. We consider two cases:\\

  $\mathbf{Case ~1. }$ If $G$ has no $ET$-partition for every subgroup $T$ of $G$.  Then according to  Theorem \ref{T1},  $G$
  is minimal non-nilpotent group.  Suppose that $G=Q\ltimes P$. First note that $S$ is cyclic, since otherwise $S$ has an $ES$-partition, which is impossible.  We consider two cases:

  If $P$ is abelian,  then clearly
 $G$ is minimal non-abelian, since $G$ is minimal non-nilpotent group, every Sylow subgroup of each proper subgroup of $G$ is abelian. If   $P$ is a proper subgroup of $G$, then it follows that $P$ is an elementary abelian subgroup and so the subgroup $Z_p\times Z_p$ has no $ES$-partition with $S\neq 1$, a contrary.  If $G=P$, that is, $G$ is a minimal non-abelian finite $p$-group. It follows that $G$ is minimal non-cyclic $p$-group and so G is a generalized quaternion group. Now as every subgroup of $G$ is cyclic or generalized quaternion group, we can obtain that $G$ is a quaternion group of order $8$, a contrary.

 If $P$ is not abelian. If $P$ has an abelian subgroup which is not cyclic, then it is easy to see that $P$ has a subgroup isomorphic to $Z_p\times Z_p$, which is a contradiction, since $Z_p\times Z_p$ has no $ES$-partition with $S\neq 1$. So we may assume that all abelian subgroups of $P$ are cyclic, then it is well-known that $P$ is generalized quaternion group. It is not hard to see that  $P$ has a subgroup isomorphic to $K=Q_8$, the quaternion group of order $8$.  Now as $S\neq 1$, $K$ has only one $ES$-partition where $S$ is $Z(K)$, which coincides with $Z(P)$ ($S=Z(K)=Z(P)$). We know that $P/S \cong D_{2^{n-1}}$, where $D_{2^{n-1}}$ is the dihedral group of order $2^{n-1}$. On the other hand, according to  Isaacs' Theorem, $P/S$ is a $p$-group of exponent $p$, so $n\in \{1,2,3\}$. It follows that $G=Q\ltimes Q_8=Z_n \ltimes Q_8$. Now as $Aut(Q_8)\cong S_4$, we can obtain that $n=2$ or $3$. Since G is not nilpotent,  $n\neq 2$ and so $G=Z_3 \ltimes Q_8$.\\

$\mathbf{Case ~2. }$ If $G$ has an arbitrary $ET$-partition for some subgroup $T$ of $G$, then by Theorem \ref{T1}, $G$ is nilpotent. Let $G=P_1\times \cdots \times P_t$, where $P_i$ is the sylow $p_i$-subgroup of $G$. As $G$ is not minimal non-cyclic group, so there exists a non-cyclic sylow $p_i$-subgroup $P_i$ of $G$, so $P_i$ has an $ES$-partition and $S\leq P_1$. If there exists $P_j$, with $j\neq i$ and $P_j$ is non-cyclic, then $S\leq P_i\cap P_j=1$, so $S=1$, a contradiction. Therefore $P_1$ is the only non-cyclic Sylow subgroup of $G$. By an argument similar to the one in the Case $(1)$, we can obtain that $P_1=Q_8$. Therefore $G=Q\times Q_8=Z_n \times Q_8$. We claim that $n$ is a prime number. Suppose, by contrary, that   $n\neq q$, $x\in Q$ and $|x|=q$. Put $K=\langle x \rangle \times P=\langle x \rangle \times Q_8$ and assume that $\{H_1, H_2, \dots, H_t\}$ is an $ET$-partition for K.  It is easy to see that $K$ has no $EZ(Q_8)$-partition while $Q_8$ has only an $E(Z(Q_8))$-partition, a contradiction. Thus $G=Z_p\times Q_8$.

As for the converse, it is enough to note that  the only non-cyclic subgroup of $Z_3 \ltimes Q_8$ is $Q_8$, which has only an $E(Z(Q_8))$-partition.  The proof is now complete.

\begin{cor}
If $G$ is a finite group with odd order and $F$ is a nontrivial subgroup of  $G$. Then  every proper non-cyclic subgroup of $G$ has  an $EF$-partition if and only if  G is minimal non-cyclic group.
\end{cor}

\bibliographystyle{amsplain}

\end{document}